\documentclass[12pt, reqno]{amsart}
\usepackage{amsmath, amsthm, amscd, amsfonts, amssymb, enumerate, graphicx, color}
\usepackage[bookmarksnumbered, colorlinks, plainpages]{hyperref}
\usepackage{fancyhdr}
\usepackage{dsfont}
\usepackage{cite}
\hypersetup{colorlinks=true,linkcolor=red, anchorcolor=green, citecolor=cyan, urlcolor=red, filecolor=magenta, pdftoolbar=true}
\newtheorem{thm}{Theorem}[section]

\newtheorem{lem}[thm]{Lemma}

\theoremstyle{definition}
\theoremstyle{remark}
\newtheorem*{Thm}{{\bf Theorem A}}

\numberwithin{equation}{section}

\newcommand{\real}{{\mathbb R}}
\newcommand{\norm}[1]{\left\Vert#1\right\Vert}
\newcommand{\8}{\infty}
\renewcommand{\a}{\mathfrak{a}}

\usepackage{comment}

\begin{document}

\title[vector-valued $q$-variational inequalities...]{vector-valued $q$-variational inequalities for averaging operators and Hilbert transform}

\thanks{{\it 2000 Mathematics Subject Classification:} Primary: 42B20, 42B25. Secondary: 46E30}
\thanks{{\it Key words:} variational inequalities, averaging operators, Hilbert transform, martingale cotype $q$, UMD}

\author{Guixiang Hong}%$^*$}
\address{School of Mathematics and Statistics, Wuhan University, Wuhan 430072 and Hubei Key Laboratory of Computational Science, Wuhan University, Wuhan 430072, China}
\email{guixiang.hong@whu.edu.cn}

\author{Wei Liu}
\address{School of Mathematics and Statistics, Wuhan University, Wuhan 430072, China}
\email{Wl.math@whu.edu.cn}

\author{Tao Ma}
\address{School of Mathematics and Statistics, Wuhan University, Wuhan 430072, China }
\email{tma.math@whu.edu.cn}

\date{}
\begin{abstract}
Recently, in \cite{GXHTM}, the authors established $L^p$-boundedness of vector-valued $q$-variational inequalities for averaging operators which take values in the Banach space satisfying martingale cotype $q$ property. In this paper, we prove that martingale cotype $q$ property is also necessary for the vector-valued $q$-variational inequalities, which is a question left open. Moreover, we characterize UMD property and martingale cotype $q$ property in terms of vector valued $q$-variational inequalities for Hilbert transform.
\begin{comment}
the $X$-valued $q$-variational inequalities for averaging operators, where the Banach space $X$ satisfies the martingale cotype $q$ properties. We also prove that the martingale cotype $q$ is necessary for the vector valued $q$-variational inequalities.
relationship between vector-valued $q$-oscillational inequalities and geometric properties of Banach spaces. We show that the Martingale cotype $q$ is necessary for the vector-valued $q$-oscillational inequalities.
\end{comment}
\end{abstract}
\maketitle

\bigskip
%%%%%%%%%%%%%%%%%%%%%%%%%%%%%%%%%%%%%%%%%%%%%%%%%%%%%%%%%%%%%%%%%%%%%%%%%%%%%%%%%%%%%%%%%%%%%%%%%%%%%%%%%%
%%%%%%%%%%%%%%%%%%%%%%%%%%%%%%%%%%%%%%%%%%%%%%%%%%%%%%%%%%%%%%%%%%%%%%%%%%%%%%%%%%%%%%%%%%%%%%%%%%%%%%%%%%
\section{Introduction}\label{ST1}

It is well known that in the setting of vector-valued harmonic analysis, many results are related to the properties of geometry of Banach spaces.
In particular, the so-called UMD, and convexity properties of Banach spaces have become the most fundamental tools in the study of vector-valued harmonic analysis. The boundedness of vector-valued singular integral operators is an important problem in this area of research. Burkholder \cite{Burk} first proved that whenever a Banach space $X$ has the UMD property, the Hilbert transform is bounded on $L^p(\real^d;X)$ for any $1<p<\8$. Later, Bourgain \cite{Bour1} established that the UMD property is indeed necessary for the boundedness of the Hilbert transform. Hyt\"{o}nen \cite{TPH} proved that two-sided Littlewood-Paley-Stein $g$-function estimate is equivalent to the UMD property of Banach space $X$. Moreover, martingale cotype $q$  plays an important role in vector-valued analysis. The martingale cotype $q$  was introduced by Pisier \cite{G.P} to characterise the properties of uniformly convex Banach spaces. In \cite{Xu}, Xu provided a characterization of the martingale cotype $q$ property via the vector-valued Littlewood-Paley theory associated to the Poisson kernel on the unit circle. Later on, Mart\'{\i}nez et al \cite{MTX} further characterized this property through general Littlewood-Paley-Stein theory. We refer \cite{TJVW} for an extensive study on vector-valued harmonic analysis.

On the other hand, variational inequalities have gained a lot of attention in past decades for many reasons, one of which is that we can measure the speed of convergence for the family of operators under consideration. In 1976, L\'{e}pingle \cite{DL} applied the regularity of Brownian motion to obtain the first scalar-valued variational inequality for martingales. Ten years later, motivated by Banach space geometry, Pisier and Xu \cite{G.P1} provided another proof of L\'epingle's result based on stopping time argument. Almost in the same time Bourgain \cite{Bour2} used L\'{e}pingle's results to establish variational inequalities for the ergodic averages of a dynamical system. Bourgain's work and Pisier-Xu's method have inspired a new research direction in ergodic theory and harmonic analysis. In \cite{JKRW,JKRW1,JKRW2}, Jones et al. proved variational inequalities for ergodic averages. Campbell et al. in \cite{JRKM} established the variational inequalities for Hilbert transform and averaging operators, and  truncated singular integrals with the homogeneous type kernel \cite{JRKM1}.  Recently, Ding et al. \cite{DHL} obtained the variational inequalities for truncated singular integrals with the rough type kernel. We recommend \cite{J.M.E.B, J.M.E.B1, JKRW3, CDHL, JRKM1, YCC, AJ, TMC, RAJ, RG, MA, RATCJ, MSB, MX, MBP, GXHTM1, MA1, K-Zk,TMI} and references therein for related work.

Among these works, some weighted norm and Banach lattice-valued variational inequality in harmonic analysis have also been established. It is worthy to mention that the first and third author \cite{GXHTM} obtained the Banach space-valued variational inequalities for ergodic averages or averaging operators whenever the underlying Banach spaces are of martingale cotype $q$, while Hyt\"{o}nen et al \cite{TMI} proved variational inequality for vector valued Walsh-Fourier series when the Banach spaces belong to a subclass of UMD spaces verifying tile-type $q$. In these two papers, finite cotype (or type) is also shown necessary for variational inequalities. However, regarding the sharp necessary conditions, they are still known to the best of the authors' knowledge.

To further motivate the purpose of the present paper, let us recall Pisier-Xu's result \cite{G.P1}. Let us start with some notations.
Let $X$ be a Banach space and $\a_t(x):(0,\8)\times\real^d\rightarrow X$ be a vector-valued function.  %let us introduce the vector-valued $q$-variational operator. Let $X$ be a Banach space and  %$q$-variation seminorm $V_q$ of a complex-valued function $(0,\8)\times\real^d \ni(t,x)\mapsto \a_t(x)$ is defined by
For $q\in[1,\8)$, the seminorm $V_q$ is defined by
$$V_q(\a_t(x):t\in Z)=\sup_{\substack{0<t_0<\cdots<t_J\\ t_j\in Z}}\bigg(\sum_{j=0}^J\|\a_{t_{j+1}}(x)-\a_{t_j}(x)\|^q\bigg)^{1/q},$$
where $Z$ is a subset of $(0,\8)$ and the supremum runs over all finite increasing sequences in $Z$.
Let $(\Sigma_n)_{n\geq 0}$ be an increasing sequence of $\sigma$-algebras on a probability space $(\Omega,\Sigma,\mathbb{P})$. The associated conditional expectations are denoted by $(\mathbb{E}_n)_{n\geq0}$. We say that the Banach space $X$ is of martingale cotype $q$ with $2\leq q<\8$, if there exists an absolute constant $C_{q}$ such that for all $f\in L^q(\Omega;X)$
\begin{equation*}%\label{martingale cotype}
\sum_{n\geq 0}\norm{\mathbb{E}_{n}f-\mathbb{E}_{n-1}f}_{L^q(\Omega;X)}^{q}\leq C_{q}\norm{f}_{L^q(\Omega;X)}^{q},%\quad 2\leq q_{0}<\infty,
\end{equation*}%
with convention $\mathbb{E}_{-1}f=0$.

In \cite{G.P1}, Pisier and Xu proved the following vector-valued variational inequalities for martingales.
%\addtocounter{theorem}{-1}
%\renewcommand{\thetheorem}{\arabic{theorem}$A$}
\begin{Thm}%\label{Thm-martingale}
Let $X$ be a Banach space.
\begin{enumerate}[\indent (i)]
  \item  If $X$ is of martingale cotype $q_0$ with $2\leq q_0<\infty$, then for every $1<p<\8$ and $q>q_0$, there exists a positive constant $C_{p,q,X}$
such that
\begin{equation}\label{variational inequality for martingale}
  \|V_q(\mathbb{E}_kf:k\in\mathbb{N})\|_{L^p}\le C_{p,q,X}\|f\|_{L^p} ,~\forall f\in L^p(\Omega;X);
\end{equation}
  \item Suppose for some $p\in(1,\8)$ and $q\in[2,\8)$ the inequality \eqref{variational inequality for martingale} holds for a Banach space $X$. Then $X$ is of martingale cotype $q$.%Let $1<p<\8$ and $2\le q<\8$, if the inequality \eqref{variational inequality for martingale} holds for some $p$, $q$ and Banach space $X$, then $X$ is of martingale cotype $q$.
\end{enumerate}
\end{Thm}

Note that the above conclusion (ii) is quite easy in the martingale setting. Indeed, using standard extrapolation argument, we can assume that the inequality \eqref{variational inequality for martingale} holds for $p=q$. Then the conclusion follows trivially from the definitions. The situation changes dramatically in harmonic analysis.

Let $f: \real \rightarrow X$ be locally integrable. Let $B_t$ be an open ball centered at the origin and its radius equal to $t>0$. Then the central averaging operators are defined as follow
\begin{equation}\label{differential operator}
 A_tf(x)=\frac{1}{|B_t|}\int_{B_t}f(x-y)dy=\frac{1}{|B_t|}\int_{\real^d}f(y)\mathds{1}_{B_t}(x-y)dy,~x\in\real,
\end{equation}
where $|\cdot|$ denotes the Lebesgue measure.%

As mentioned before, the  authors in \cite[Theorem 2.1]{GXHTM} obtained the $L^p$-boundedness of vector-valued variational inequalities for averaging operators when the underlying Banach spaces satisfy martingale cotype property. But they could not obtain results similar to (ii) of Theorem A. In this paper, we do it.

\begin{thm}\label{Thm:first}
Let $X$ be a Banach space.
\begin{enumerate}[\indent (i)]
\item If $X$ is of martingale cotype $q_0$ with $2\leq q_0<\infty$. Then for any $1<p<\infty$ and $q_0<q<\infty$, there exists a constant $C_{p,q}$ such that
\begin{equation}\label{strong(p,p)inequalities}
\norm{V_q(A_tf:t>0)}_{L^p}\leq C_{p,q}\norm{f}_{L^p},\;\forall f\in L^{p}(\real;X);
\end{equation}
\item  Let $2\leq q<\8$, and for some  $p\in(1,\8)$, there exists a positive constant $C_{p,q}$ such that
\begin{equation}\label{first-inequality}
 \norm{V_q(A_tf:t>0)}_{L^p}\leq C_{p,q}\norm{f}_{L^p},~\forall f\in L^{p}(\real;X),
\end{equation}
then $X$ is of martingale cotype $q$.
\end{enumerate}
\end{thm}

Conclusion (i) has been obtained in \cite{GXHTM}. To show conclusion (ii), we use several times the transference techniques and follows essentially Xu's argument in his remarkable paper \cite{Xu} based on Bourgain's transference method. Many modifications are surely necessary in the present setting, see Section  \ref{ST2} for details.

Before stating our second result, we first introduce some notations. % is the characterization of UMD property through vector-valued variational inequalities for Hilbert transform.
Recall that a Banach space $X$ is said to have the UMD property if for all scalars $r_n = \pm1, ~n=1,\cdots,N$, there exists a positive constant independent of $N$ such for all $1<p<\8$ and all $X$-valued $L^p$-martingale differences $(df_n)_{n\geq1}$,
\begin{equation*}
  \|\sum_{n=1}^Nr_ndf_n)\|_{L^p(\Omega;X)}\le C_{p,X}\|\sum_{n=1}^Ndf_n\|_{L^p(\Omega;X)}.
\end{equation*}
Let $\mathcal{S}(\real)$ be the set of Schwartz functions on $\real$. Define $f\in \mathcal{S}(\real)\bigotimes X\subseteq L^{p}(\real;X)$ by setting $f(x)=\sum^{n}_{k=1}\varphi_{k}(x)b_k$, where $\varphi_1,...,\varphi_n\in \mathcal{S}(\real)$ and $b_1,...,b_n\in X$. For $f\in \mathcal{S}(\real)\otimes X$,  the truncated Hilbert transform is defined by
$$H_{\epsilon} f(x)=\frac{1}{\pi}\int_{|y|>\epsilon}\frac{f(x-y)}{y}dy,$$
%and the Hilbert transform is defined as
%$$Hf(x)=\lim\limits_{\epsilon\downarrow 0,~R\rightarrow\8}H_{\epsilon,R} f(x),$$

%if the limit exists for almost every $x\in\real$. When $R=\8$, we simply denote $H_{\epsilon,\8}$ by $H_\epsilon$, which is nothing but the usual truncated Hilbert transform .

The second result of this paper is on vector-valued variational inequalities for Hilbert transform.

\begin{thm}\label{Thm:second}
Let $X$ be a Banach space.
\begin{enumerate}[\indent (i)]
\item Let $2\leq q_0<\infty$, and $X$ be of martingale cotype $q_0$ satisfying UMD property. Then for any $1<p<\infty$ and $q_0<q<\infty$, there exist a constant $C_{p,q}$ such that
\begin{equation}\label{Ineq:first}
\norm{V_q(H_\epsilon f:\epsilon>0)}_{L^p}\leq C_{p,q}\norm{f}_{L^p},\;\forall f\in L^{p}(\real;X);
\end{equation}
\item  Let $2\leq q<\8$. If for some $p\in(1,\8)$, there exists a positive constant $C_{p,q}$ such that
\begin{equation}\label{Ineq:second}
 \norm{V_q(H_\epsilon f:\epsilon>0)}_{L^p}\leq C_{p,q}\norm{f}_{L^p},~\forall f\in L^{p}(\real;X),
\end{equation}
then $X$ satisfies UMD property and is of martingale cotype $q$.
\end{enumerate}
\end{thm}

With Theorem \ref{Thm:first} at hand, the proof of Theorem \ref{Thm:second} is relatively easy but some of the argument seems completely new. We use again several times the transference techniques for both (i) and (ii), see Section \ref{ST3} for details.

Throughout this paper, by $C$ we always denote a positive constant that may vary from line to line.

\bigskip
%%%%%%%%%%%%%%%%%%%%%%%%%%%%%%%%%%%%%%%%%%%%%%%%%%%%%%%%%%%%%%%%%%%%%%%%%%%%%%%%%%%%%%%%%%%%%%%%%%%%%%%%%%%%%%%%%%%%%%%%%%%
%%%%%%%%%%%%%%%%%%%%%%%%%%%%%%%%%%%%%%%%%%%%%%%%%%%%%%%%%%%%%%%%%%%%%%%%%%%%%%%%%%%%%%%%%%%%%%%%%%%%%%%%%%%%%%%%%%%%%%%%%%%
\section{the proof of Theorem \ref{Thm:first}}\label{ST2}

In this section, we follow closely Xu's argument \cite{Xu} to prove Theorem~\ref{Thm:first}. The main ingredient is Bourgain's transferenece method. We first introduce some notations.

Let $\mathbb{P}_t(\theta)=\sum_ne^{-t|n|}e^{2\pi in\theta}$ be the Poisson kernel of the unit disc. For $f\in L^p([0,1];X)$, denote the Poisson integral

$$\mathbb{P}_tf(\xi)=\mathbb{P}_t\ast f(\xi)=\int_0^1f(e^{2\pi i\theta})\bigg(\sum_ne^{-t|n|}e^{2\pi in(\xi-\theta)}\bigg)d\theta.$$

Let $m$ be a positive integer and $1\le k\le m$. Let $a_k(e^{2\pi i\theta_1},\cdots, e^{2\pi i\theta_{k-1}})$ be a trigonometric polynomial defined on $[0,1]^{k-1}$ with value in Banach space $X$,
\begin{align*}
  a_{k}&(e^{2\pi i\theta_1},\cdots,e^{2\pi i\theta_{k-1}})\\
  &=\sum_{|m_1|\le N_1}\cdots \sum_{|m_{k-1}|\le N_{k-1}} x_{m_1,\cdots, m_{k-1}}e^{2\pi im_1\theta_1}\cdots e^{2\pi im_{k-1}\theta_{k-1}},
\end{align*}
with convention $a_k=x_0$, when $k=1$. Let $b_k$ be a trigonometric polynomial with complex coefficients such that $\hat{b}(0)=0$,
\begin{equation*}
b_{k}(e^{2\pi i\theta_{k}})=\sum_{1\le|j|\le M_k}y_je^{2\pi ij\theta_{k}}.
\end{equation*}%
Let $\{n_k\}_{1\le k\le m}$ be a sequence of positive integers. Set%&=a_{k,(n)}(e^{2\pi i\theta};e^{2\pi i\theta_1},\cdots,e^{2\pi i\theta_{m}})
\begin{align*}
 a_{k,(n)}(\theta)&=a_k(e^{2\pi i(\theta_1+n_1\theta)},\cdots,e^{2\pi i(\theta_{k-1}+n_{k-1}\theta)}),\\
 b_{k,(n)}(\theta)&=b_{k}(e^{2\pi i(\theta_k+n_k\theta)}),\\
f(\theta)&=\sum_{k=1}^{m}a_{k,(n)}(\theta) b_{k,(n)}(\theta)\\
&=\sum_{k=1}^{m}f_{k,(n)}(\theta).
\end{align*}
In the following we fix $\theta_1, \cdots,\theta_m$. All these functions are considered as trigonometric polynomials in $e^{i\theta}$. %whereas $\theta_1, \cdots,\theta_m$ are fixed.%
By the  definition of $a_{k,(n)}(\theta)$ and $b_{k,(n)}(\theta)$, we have
\begin{equation*}%\label{tri-function}
  \begin{split}
   f_{k,(n)}(\theta)&=\sum_{|m_1|\le N_1,\cdots, |m_{k-1}|\le N_{k-1}}\sum_{1\le|j|\le M_k}\big(x_{m_1,\cdots,m_{k-1}}y_j\\
   &\times e^{2\pi i(m_1\theta_1+\cdots+m_{k-1}\theta_{k-1}+j\theta_{k})}e^{2\pi i(m_1n_1+\cdots+m_{k-1}n_{k-1}+jn_{k})\theta}\big);
  \end{split}
\end{equation*}
and
\begin{equation}\label{weak poisson integral}
  \begin{split}
     \mathbb{P}_{t}f_{k,(n)}(\theta)=\sum_{m_i}\sum_{j}\big(&x_{m_1,\cdots,m_{k-1}}y_je^{2\pi i(m_1\theta_1+\cdots+m_{k-1}\theta_{k-1}+j\theta_{k})}\times\\
     &e^{-|m_1n_1+\cdots+m_{k-1}n_{k-1}+jn_{k}|t}e^{2\pi i(m_1n_1+\cdots+m_{k-1}n_{k-1}+jn_{k})\theta}\big).
\end{split}
\end{equation}

The following lemma is a discrete version of Xu's \cite{Xu} Lemma 3.5.
\begin{lem}\label{select lemma}
Let $a_{k,(n)}$, $b_{k,(n)}$ and $f_{k,(n)}$ are defined as above. For any $\epsilon>0$, we can choose an increasing sequence $\{n_k\}_{1\leq k\leq m}$ of positive integers and a decreasing sequence $\{{l_k}\}_{0\le k\le {m}}\subset [0,+\8]$ such that for all $\theta\in [0,1]$

\begin{equation}\label{ineq-1}
\|\mathbb{P}_{t}f_{k,(n)}(\theta)\|<\frac{\epsilon}{2^k}, ~\forall~t\geq {l_{k-1}},
\end{equation}
and for any $0\le t\le l_k$
\begin{equation}\label{ineq-2}
\sum_{j=1}^{k}\|\mathbb{P}_{l_k}f_{j,(n)}(\theta)-\mathbb{P}_{t}f_{j,(n)}(\theta)\|<\epsilon.
\end{equation}%~\forall~t\geq 2^{l_{k-1}},
\end{lem}
\begin{proof}

First we set $l_0=+\8$ and $n_1=1$. Suppose we have already defined $+\8=l_0>l_1>\cdots>l_{k-1}$ and $1=n_1<\cdots<n_k$ with the required properties. Now, we define $l_k$ and $n_{k+1}$ as the following.

Since $f$ is a trigonometric polynomial, then $\mathbb{P}_t\ast f$ uniformly convergence to $f$ as $t\rightarrow 0^+$ on $\theta\in [0,1]$. Hence we can choose an $l_k$ sufficiently close to $0$ and for any $0<t<l_k$ (\ref{ineq-2}) holds. Next, we choose $n_{k+1}$ such that (\ref{ineq-1}) at the step $k+1$ holds.

According to the definition $f_{k,n}$,  we know that there is a finite integer $N$ such that $|m_1n_1+\cdots+m_kn_k|<N$ for all $|m_1|\le N_1,\cdots, |m_k|\le N_k$. On the other hand, since $\hat{b}_{k+1}(0)=0$, if $|n_{k+1}|$ is sufficiently big (precisely, $n_{k+1}\in \mathbb{Z}\setminus[-N,N]$), then $|m_1n_1+\cdots+m_kn_k+jn_{k+1}|>0$ for all $|m_1|\le N_1,\cdots, |m_k|\le N_k$ and $1\le |j|\le M$. Then by (\ref{weak poisson integral}), we can choose $n_{k+1}$ such that (\ref{ineq-1}) holds at step $k+1$.

Finally, let $l_m=0$, then we have set up the required sequences $(n_k)_{1\le k\le m}$ and $(l_k)_{0\le k\le{m}}$.
\end{proof}

The following lemma is cited from \cite{GXHTM}.

\begin{lem}\label{q-variation for poisson kernel}
Let $X$ be a Banach space.  If inequality \eqref{first-inequality} holds for some $p$, then for every $1<p<\8$ we have
\begin{equation}\label{inequality for poisson kernel}
  \|V_q(\mathbb{P}_tf:t>0)\|_{L^p}\le C_{p,q}\|f\|_{L^p},~\forall f\in L^p([0,1];X).
\end{equation}
\end{lem}

By the method of the proof of Theorem 3.1 in \cite{Xu}, we now prove Theorem~\ref{Thm:first}.
\begin{proof}[Proof of Theorem \ref{Thm:first}]
The first and third author proved the part (i) in \cite[Theorem 2.1]{GXHTM}. Hence, we only need to prove the part (ii).

Let $\Omega=\{-1,1\}^{\mathbb{N}}$ be the dyadic group equipped with its normalized Haar measure $P$. Let $m$ be a positive integer.  According to \cite{G.P}, we know that to prove the Banach space $X$ is of martingale cotype $q$, it suffices to prove that there exists a constant $C$ such that for all finite  Walsh-Paley martingale $M=(M_k)_{1\leq k\leq m}$  %Before prove Theorem \ref{Thm:first}, we need some lemmas.%
\begin{equation}\label{dyadic martingale cotype}
\sum_{k=1}^m\mathbb{E}\|M_k-M_{k-1}\|^q\leq C_q\sup_{k}\mathbb{E}\|M_k\|^q,
\end{equation}
where denote $M_{-1}=0$.

Fix such a martingale $M=(M_k)_{1\le k\le m}$. Let $(\varepsilon_k)_{1\le k\le m}$ be the coordinate sequence of $\Omega$. By the representation of Walsh-Paley martingales (cf. \cite{TJVW}), we know that there exist functions $\phi_k:\{-1,1\}^{k-1}\rightarrow X$ such that
$$M_k-M_{k-1}=\phi_k(\varepsilon_1,\cdots,\varepsilon_{k-1})\varepsilon_k,~1\le k\leq m,$$
where $\varepsilon_k\in\{-1,1\}^k$ and $\phi_1$ is a constant belong to $X$.

Following \cite{Bour}, we use transference from $\Omega$ to the group $[0,1]^\mathbb{N}$. For $1\leq k\leq m$, set
\begin{equation*}
\begin{split}
&a_k(e^{2\pi i\theta_1},\cdots,e^{2\pi i\theta_{k-1}})=\phi_k(sgn(cos2\pi\theta_1),\cdots,sgn(cos2\pi\theta_{k-1})),\\
&b_k(e^{2\pi i\theta_k})=sgn(cos2\pi\theta_k).
\end{split}
\end{equation*}

Since $a_k(e^{2\pi i\theta_1},\cdots,e^{2\pi i\theta_{k-1}})\in L^q([0,1]^{k-1};X)$ and $b_k(e^{2\pi i\theta_k})\in L^1([0,1])$, %For $k=1$, $\phi_1\in L^q([0,1];X)$ and $sgn(cos\theta_1)\in L^1([0,1])$.
by Ces\`{a}ro summability, we know that for any $\epsilon>0$, there are two trigonometric polynomials $a^\prime_k$ and $b^\prime_k$ with values respectively in $X$ and $\mathbb{C}$ such that $\|a^\prime_k-a_k\|_q<\epsilon$ and $\|b^\prime_k-b_k\|_q<\epsilon$. Note that $\int_0^1{sgn(cos2\pi\theta)}d\theta=0$. Then $\hat{b^\prime_k}(0)=0$. By this consideration, we assume that $a_k$ and $b_k$ are trigonometric polynomials with values respectively in $X$ and $\mathbb{C}$ and $\hat{b_k}(0)=0$, such that $\|f_{k,(n)}-(M_k-M_{k-1})\|_q<\epsilon$.% and $\|\|<\epsilon$.

\medskip

In the following, we keep all notations introduced in Lemma \ref{select lemma}. By Lemma~\ref{select lemma} we have a sequence $\{n_i\}_{1\le i\le m}$ and $\{l_i\}_{0\le i\le m}$ such that \eqref{ineq-1} and \eqref{ineq-2} hold.

Firstly, for $i=1$, we have
\begin{equation*}
  \|\sum_{k=1}^{m}(\mathbb{P}_{l_1}f_{k,(n)}-\mathbb{P}_{l_{0}}f_{k,(n)})-f_{1,(n)}\|\leq\|
  \mathbb{P}_{l_1}f_{1,(n)}-f_{1,(n)}\|+\|\sum_{k=2}^{m}\mathbb{P}_{l_1}f_{k,(n)}\|,
\end{equation*}
combining \eqref{ineq-2} with \eqref{ineq-1}, we obtain the following estimate
\begin{equation*}
  \begin{split}
  \|\mathbb{P}_{l_1}f_{1,(n)}-f_{1,(n)}\|<\epsilon,~\|\sum_{k=2}^{m}\mathbb{P}_{l_1}f_{k,(n)}\|\le\sum_{k=2}^{m}\frac{\epsilon}{2^k}<\epsilon.
   \end{split}
\end{equation*}
Hence for $i=1$, we get
\begin{equation*}%\label{the estimate-1}
 \|\sum_{k=1}^{m}(\mathbb{P}_{l_1}f_{k,(n)}-\mathbb{P}_{l_{0}}f_{k,(n)})-f_{1,(n)}\|\leq \epsilon.
\end{equation*}
For $1<i\le m$, we first have the following inequality
\begin{align*}
  &\|\sum_{k=1}^{m}(\mathbb{P}_{l_i}f_{k,(n)}-\mathbb{P}_{l_{i-1}}f_{k,(n)})-f_{i,(n)}\|\\
  &\leq\|\sum_{k=1}^{i-1}(\mathbb{P}_{l_i}f_{k,(n)}-\mathbb{P}_{l_{i-1}}f_{k,(n)})\|+\|\sum_{k=i}^{m}
  (\mathbb{P}_{l_i}f_{k,(n)}-\mathbb{P}_{l_{i-1}}f_{k,(n)})-f_{i,(n)}\|\\
  &=I+II.
\end{align*}
Now, we estimate $I$ and $II$. By (\ref{ineq-2}), the first term $I$ is controlled by
$$%\|\sum_{k=1}^{i-1}(\mathbb{P}_{l_i}f_{k,(n)}-\mathbb{P}_{l_{i-1}}f_{k,(n)})\|
I<\sum_{k=1}^{i-1}\|\mathbb{P}_{l_i}f_{k,(n)}-\mathbb{P}_{l_{i-1}} f_{k,(n)}\|<\epsilon.$$
To handle part $II$, one can observe the following inequality
\begin{align*}
 %\|\sum_{k=i}^{m}(\mathbb{P}_{l_i}f_{k,(n)}-\mathbb{P}_{l_{i-1}}f_{k,(n)})\|
 II&\le \|\sum_{k=i+1}^{m}\mathbb{P}_{l_i}f_{k,(n)}\|+\|\mathbb{P}_{l_i}f_{i,(n)}-f_{i,(n)}\|+\|\sum_{k=i}^{m}\mathbb{P}_{l_{i-1}}f_{k,(n)}\|\\
 &=III+IV+V.
\end{align*}
For $III$ and $V$, by (\ref{ineq-1}) we obtain
$$III\le \sum_{k=i}^{m}\|\mathbb{P}_{l_i}f_{k,(n)}\|\leq \sum_{k=i+1}^{m}\frac{\epsilon}{2^k}<\epsilon, ~\textit{and}, ~ V\le \sum_{k=i}^{m}\|\mathbb{P}_{l_{i-1}}f_{k,(n)}\|\leq \sum_{k=i}^{m}\frac{\epsilon}{2^k}<\epsilon.$$
For $IV$, using (\ref{ineq-2}) again, we have
$$IV=\|\mathbb{P}_{l_i}f_{i,(n)}-f_{i,(n)}\|<\epsilon.$$
Therefore, we have $\|\sum_{k=1}^{m}(\mathbb{P}_{l_i}f_{k,(n)}-\mathbb{P}_{l_{i-1}}f_{k,(n)})-f_{i,(n)}\|\le 3\epsilon$. By this estimate we see that
\begin{align*}
&\int_{[0,1]^m}\int_{[0,1]}\|f_{i,(n)}(\theta)\|^qd\theta d\theta_1\cdots d\theta_m\\
&\le C_q\int_{[0,1]^m}\int_{[0,1]}\bigg\|\sum_{k=1}^{m}(\mathbb{P}_{l_i}f_{k,(n)}(\theta)-\mathbb{P}_{l_{i-1}}f_{k,(n)}(\theta))\bigg\|^q
 d\theta d\theta_1\cdots d\theta_m\\
&+C_q\int_{[0,1]^m}\int_{[0,1]}\bigg\|\sum_{k=1}^{m}(\mathbb{P}_{l_i}f_{k,(n)}(\theta)-\mathbb{P}_{l_{i-1}}f_{k,(n)}(\theta))-f_{i,(n)}(\theta)\bigg\|^q
 d\theta d\theta_1\cdots d\theta_m\\
 &\le C_{q}\int_{[0,1]^m}\int_{[0,1]}\bigg\|\sum_{k=1}^{m}(\mathbb{P}_{l_i}f_{k,(n)}(\theta)-\mathbb{P}_{l_{i-1}}f_{k,(n)}(\theta))\bigg\|^qd\theta d\theta_1\cdots d\theta_m+C_q\epsilon^q.
\end{align*}
Hence we can observe that
\begin{equation}\label{control ineq}
  \begin{split}
     &\sum_{i=1}^m\int_{[0,1]^m}\int_{[0,1]}\bigg\|f_{i,(n)}(\theta)\bigg\|^qd\theta d\theta_1\cdots d\theta_m\\
  &\le C_{q}\sum_{i=1}^m\int_{[0,1]^m}\int_{[0,1]}\bigg\|\sum_{k=1}^{m}(\mathbb{P}_{l_i}f_{k,(n)}(\theta)
  -\mathbb{P}_{l_{i-1}}f_{k,(n)}(\theta))\bigg\|^qd\theta d\theta_1\cdots d\theta_m\\
  &+C_qm\epsilon^q.
  \end{split}
\end{equation}
On the other hand, by our assumption that inequality \eqref{first-inequality} holds, and Lemma~\ref{q-variation for poisson kernel}, we have
\begin{equation}\label{leq-equality}
  \begin{split}
 \int_{[0,1]^m}\int_{[0,1]}\sum_{i}\bigg\|\sum_{k=1}^{m}(\mathbb{P}_{l_i}f_{k,(n)}(\theta)-\mathbb{P}_{l_{i-1}}f_{k,(n)}(\theta))\bigg\|^qd\theta d\theta_1\cdots d\theta_m\\
 \le C_{q}\int_{[0,1]^m}\int_{[0,1]}\bigg\|\sum_{k=1}^{m}f_{k,(n)}(\theta)\bigg\|^qd\theta d\theta_1\cdots d\theta_m.
\end{split}
\end{equation}
Combining \eqref{control ineq} with \eqref{leq-equality}, we have
\begin{equation*}
\begin{split}
 \sum_{i=1}^m&\int_{[0,1]^m}\int_{[0,1]}\bigg\|f_{i,(n)}(\theta)\bigg\|^qd\theta d\theta_1\cdots d\theta_m\\
 &\le C_q \int_{[0,1]^m}\int_{[0,1]}\bigg\|\sum_{k=1}^{m}f_{k,(n)}(\theta)\bigg\|^qd\theta d\theta_1\cdots d\theta_m+C_qm\epsilon^q.
\end{split}
\end{equation*}
By the Fubini theorem and the change of variables
\begin{equation*}%\label{change of variables}
 (\theta_1,\cdots,\theta_m)\mapsto(\theta_1-n_1\theta,\cdots,\theta_m-n_m\theta),
\end{equation*}
we obtain
$$\sum_{k=1}^m\mathbb{E}\|M_k-M_{k-1}\|^q\le C_q\mathbb{E}\|M_m\|^q+C_qm\epsilon^q.$$
Letting $\epsilon\rightarrow 0$, we deduce the desired inequality (\ref{dyadic martingale cotype}). Thus, we proved that $X$ is of martingale cotype $q$. This completes the proof of the part (ii).
\end{proof}

\bigskip
%%%%%%%%%%%%%%%%%%%%%%%%%%%%%%%%%%%%%%%%%%%%%%%%%%%%%%%%%%%%%%%%%%%%%%%%%%%%%%%%%%%%%%%%%%%%%%%%%%%%%%%%%%%%%%%%%%%%%%%%%%%%%%%%%%%%
%%%%%%%%%%%%%%%%%%%%%%%%%%%%%%%%%%%%%%%%%%%%%%%%%%%%%%%%%%%%%%%%%%%%%%%%%%%%%%%%%%%%%%%%%%%%%%%%%%%%%%%%%%%%%%%%%%%%%%%%%%%%%%%%%%%%
\section{the proof of Theorem \ref{Thm:second}}\label{ST3}
In this section, we prove Theorem \ref{Thm:second}. Before proving the theorem, we need some notations and lemmas.

Let $P(x)=\frac{1}{\pi}\frac{1}{1+x^2}$ be the Poisson function in $\real$. For $\epsilon>0$, set $P_\epsilon(x)=\frac{1}{\epsilon}P(\epsilon^{-1}x)$, %and $Q_\epsilon(x)=\frac{1}{\pi}\frac{x}{x^2+\epsilon^2}$.
and define
$$P_\epsilon f(x)=P_\epsilon\ast f(x)=\int_{\real}P_\epsilon(x-y)f(y)dy.$$
Let $Q_\epsilon$ denote the conjugate Poisson kernel, that is $Q_\epsilon(x)=\frac{1}{\pi}\frac{x}{\epsilon^2+x^2}$. Define
$$Q_\epsilon f(x)=Q_\epsilon\ast f(x)=\int_{\real}Q_\epsilon(x-y)f(y)dy.$$
Then we have the following  equality
\begin{equation}\label{the equality of Poisson function}
  Q_\epsilon f(x)=P_\epsilon(Hf)(x),~Q_\epsilon(Hf)(x)=-P_\epsilon f(x).
\end{equation}

The following lemma can be obtained by the vector-valued variational inequalities for averaging operators (see \cite[Lemma 3.2]{GXHTM} or \cite[Corollary 2.6]{JKRW1}).
\begin{lem}\label{main lem}
Let $X$ be a Banach space. If $X$ is of martingale cotype $q_0$ with $q_0\ge 2$, then for every $1<p<\8$ and $q>q_0$,
\begin{equation}\label{Lp bounded of pisson function}
  \|V_q(P_\epsilon f:\epsilon>0)\|_{L^p}\le C_{p,q}\|f\|_{L^p},~\forall f\in L^{p}(\real;X).
\end{equation}
\end{lem}

%By equality~\eqref{the equality of Poisson function},
We also need the following lemma.

\begin{lem}\label{remark of poisson function}
If the Banach space $X$ satisfies UMD property, and $1<p<\8$, then the following two inequalities are equivalent:

\begin{enumerate}
  \item \begin{equation*}%\label{the equality of conjugate poisson kernel}
  \|V_q(Q_\epsilon f:\epsilon>0)\|_{L^p}\le C_{p,q}\|f\|_{L^p},
\end{equation*}
  \item \begin{equation*}%\label{Lp bounded of pisson function}
  \|V_q(P_\epsilon f:\epsilon>0)\|_{L^p}\le C_{p,q}\|f\|_{L^p}.
\end{equation*}
\end{enumerate}
\end{lem}

The following lemma extends the results of Lemma \ref{main lem} for more general functions, the proof just follows the scalar case as in \cite{JRKM}. We omit the details.

%The same argument as in Lemma 2.4 of [2] can be repeated word by word to show Lemma 3.2, since B is a Banach space. We leave the details to the interesting readers.
\begin{lem}\label{remark of poisson function1}
Let $X$ be a Banach space. Let $\Phi$ be a scalar-valued function defined on $\real$ and $f\in L^p(\real;X)$.  Define $\Phi_t(x)=\frac{1}{t}\Phi(\frac{x}{t})$, and set $\Phi_tf(x)=\Phi_t\ast f(x)=\int_{\real}\Phi_t(y)f(x-y)dy$. If the Banach space $X$ is of martingale cotype $q_0$ with $q_0\ge 2$, then for every $q>q_0$ and $1<p<\8$, we have the following:%the inequality \eqref{first-inequality} holds,
\begin{enumerate}
  \item Let $\Phi$ be supported on $[0,1]$. Assume $\Phi$ is differentiable on $(0,1)$, and that $\int_0^1x|\Phi^\prime(x)|dx<\8$, then $V_q(\Phi_tf:t>0)$ is $L^p$ bounded.
  \item Let $\Phi$ be supported on $[1,\8)$. Assume $\Phi$ is differentiable on $(0,1)$, and that $\lim_{x\rightarrow\8}\Phi(x)=0$. Moreover, $\Phi(x)$ satisfies $\int_1^\8 x|\Phi^\prime(x)|dx<\8$, then $V_q(\Phi_tf:t>0)$ is $L^p$ bounded.
\end{enumerate}
\end{lem}

\begin{proof}[Proof of Theorem \ref{Thm:second} (i)]
For $f\in \mathcal{S}(\real)\bigotimes X$, by the triangle inequality, we have
$$V_q(H_\epsilon f:\epsilon>0)\leq V_q((H_\epsilon-Q_\epsilon)f:\epsilon>0)+V_q(Q_\epsilon f:\epsilon>0).$$
By our assumption that the Banach space $X$ is of martingale cotype $q_0$ and satisfies UMD property, by Theorem \ref{Thm:first} (i), Lemma \ref{main lem} and Lemma~\ref{remark of poisson function}, it will be enough to estimate $V_q([H_\epsilon-Q_\epsilon]f(x):\epsilon>0)$.

For $f\in \mathcal{S}(\real)\bigotimes X$, we observe
\begin{align*}
&H_{\epsilon}f(x)-Q_\epsilon f(x)\\
&=\frac{1}{\pi}\int_{|y|>\epsilon}f(x-y)(\frac{1}{y}-\frac{y}{y^2+\epsilon^2})dy-\frac{1}{\pi}\int_{|y|<\epsilon}f(x-y)\frac{y}{y^2+\epsilon^2}dy\\
&=\frac{1}{\pi}\int^{\infty}_{\epsilon}f(x-y)(\frac{1}{y}-\frac{y}{y^2+\epsilon^2})dy-\frac{1}{\pi}\int^{\infty}_{\epsilon}f(x+y)(\frac{1}{y}-\frac{y}
{y^2+\epsilon^2})dy\\
&\qquad -\frac{1}{\pi}\int^{\epsilon}_{0}f(x-y)\frac{y}{y^2+\epsilon^2}dy+\frac{1}{\pi}\int^{\epsilon}_{0}f(x+y)\frac{y}{y^2+\epsilon^2}dy\\
&= A^{+}_{\epsilon}f(x)-A^{-}_{\epsilon}f(x)-G^{+}_{\epsilon}f(x)+G^{-}_{\epsilon}f(x).
\end{align*}
Denote $\phi^+(y)=\frac{1}{\pi}\frac{1}{y(y^2+1)}\mathds{1}_{[1,\8)}(y)$, $\phi^-(y)=\frac{1}{\pi}\frac{1}{y(y^2+1)}\mathds{1}_{(-\8,-1]}(y)$,
 $A^{+}_{\epsilon}f(x)=f\ast \phi^+_{\epsilon}(x)$ and $A^{-}_{\epsilon}f(x)=f\ast \phi^-_{\epsilon}(x)$. Also let $\rho^+(y)=\frac{1}{\pi}\frac{y}{y^2+1}\mathds{1}_{[0,1]}(y)$, $\rho^-(y)=\frac{1}{\pi}\frac{y}{y^2+1}\mathds{1}_{[-1,0]}(y)$,
 $G^{+}_{\epsilon}f(x)=f\ast \rho^+_{\epsilon}(x)$ and $G^{-}_{\epsilon}f(x)=f\ast \rho^-_{\epsilon}(x)$. By Lemma~\ref{remark of poisson function1}
 $\|V_q(A^+_\epsilon f:\epsilon>0)\|_{L^p}\le C_{q}\|f\|_{L^p}$ and $\|V_q(A^-_\epsilon f:\epsilon>0)\|_{L^p}\le C_{q}\|f\|_{L^p}$. Similar estimates also hold for operators $V_q(G^{+}_{\epsilon}f:\epsilon>0)$ and $V_q(G^{-}_{\epsilon}f:\epsilon>0)$. This completes to the proof of the part (i) of Theorem~ \ref{Thm:second}.
\end{proof}

Now let us prove the part (ii) of Theorem~ \ref{Thm:second}. Before proving this part, we first introduce some lemmas.

Let $R>\epsilon>0$. Clearly, the function $y^{-1}\mathds{1}_{\{\epsilon<|y|<R\}}$ belong to $L^q(\real)$, $1<q\le\8$. Hence by H\"{o}lder's inequality,
\begin{equation*}
  H_{\epsilon,R}f(x)=\frac{1}{\pi}\int_{\epsilon<|y|<R}\frac{f(x-y)}{y}dy
\end{equation*}
is well defined for all $f\in L^p(\real;X)$, $p\ge 1$. The same is true for $H_\epsilon f$. %It is well-known that variational inequalities are a powerful tool for proving pointwise convergence for the family of operators under consideration.

\begin{comment}
In classical harmonic analysis, certain weak type boundedness for the maximal function implies almost everywhere convergence (see \cite[section 2.1.3]{Graf} for more details). The following lemma is a vector-valued version.
\begin{lem}\label{maximal lemma}
Let $X$ be Banach space and $1<p<\8$. Suppose that for some $C_p>0$ and for all $f\in L^p(\real;X)$ we have
\begin{equation*}
  \|\sup_{R>\epsilon>0}\|H_{\epsilon,R}f\|\|_{L^p}\le C_p\|f\|_{L^p}
\end{equation*}
and for all $f\in\mathcal{S}(\real)\bigotimes X$,
\begin{equation*}
  \lim\limits_{\epsilon\downarrow 0,~R\uparrow\8}H_{\epsilon,R}f(x)=Hf(x),~a.e~x\in\real.
\end{equation*}
Then for all function $f$ in $L^p(\real;X)$, the Hilbert transform $H$ is well defined in $L^p(\real;X)$ and satisfies %$Hf=\lim\limits_{\substack{\epsilon\downarrow 0\\ R\rightarrow\8}}H_{\epsilon, R}f$
\begin{equation*}
  \|Hf\|_{L^p}\le C_p\|f\|_{L^p}.
\end{equation*}
%where the constant $C^_p$ independent of $f$.
\end{lem}
Since $X$ is a Banach space, then we can use the same argument as in Theorem~2.1.14 of \cite{Graf} to prove Lemma \ref{maximal lemma}. We omit the details.
\end{comment}

In the following, we will prove the main lemma in this section.
%By Lemma \ref{maximal lemma}, we have the following lemma.
%we prove the part (ii) of Theorem~ \ref{Thm:second}. Note that by Lemma~\ref{q-variation for poisson kernel}, the part (ii) of Theorem \ref{Thm:second} reduces to prove the following lemma.

\begin{lem}\label{1 main lemma}
Let $X$ be a Banach space. For $2\le q<\8$, if there exists some $1<p<\8$ and a positive constant $C_{p,q}$ such that
\begin{equation}\label{q-variation for hilbert transform}
\norm{V_q(H_\epsilon f:\epsilon>0)}_{L^p}\leq C_{p,q}\norm{f}_{L^p},\;\forall f\in L^{p}(\real;X),
\end{equation}
then $X$ has UMD property and we have the following inequality
\begin{equation}\label{variation of poisson}
  \norm{V_q(P_\epsilon f:\epsilon>0)}_{L^p}\leq C_{p,q}\norm{f}_{L^p},\;\forall f\in L^{p}(\real;X).
\end{equation}
\end{lem}

\begin{proof}
Since variational inequalities imply the pointwise convergence for the family of operators under consideration.
%By the definition of $q$-variation, it is easy to see that
%$$\sup_{R>\epsilon>0}\|H_{\epsilon,R} f\|\le V_q(H_\epsilon f:\epsilon>0).$$
%Then for $1<p<\8$, we have
%$$\norm{\sup_{R>\epsilon>0} \norm{H_{\epsilon,R} f}}_{L^p}\leq\norm{V_q(H_\epsilon f:\epsilon>0}_{L^p},$$
%Since $\sup_\epsilon \norm{H_\epsilon f(x)}\leq V_q(H_\varepsilon)f(x)$,\norm{\overline{\lim}_{\epsilon\rightarrow 0}\|H_{\epsilon,1/\epsilon}f\|}_{L^p}
Then by condition (\ref{q-variation for hilbert transform}), we have for every $f\in L^{p}(\real;X)$
\begin{equation*}
  \|H_{\epsilon,1/\epsilon}f(x)\|\rightarrow\|Hf(x)\|,~a.e.~x\in\real,~as~\epsilon\rightarrow 0.
\end{equation*}
Furthermore,

  \begin{align*}
    \norm{Hf(x)}_{L^p(\real;X)}&=\norm{\lim_{\epsilon\rightarrow 0}\|H_{\epsilon,1/\epsilon}f(x)\|_X}_{L^p(\real)}\\
    &\le\norm{V_q(H_\epsilon f(x):\epsilon>0)}_{L^p(\real;X)}\\
    &\le C_{p,q}\norm{f}_{L^p(\real;X)}.
  \end{align*}

Hence, by the $L^p$-boundedness of vector-valued Hilbert transform, we know that the Banach space $X$ is of UMD (cf.~\cite[Theorem 5.1.1]{TJVW}).%

Next we will prove inequality~\eqref{variation of poisson}. Since
\begin{align*}
  Q_\epsilon(x)&=\frac{1}{\pi}\frac{x}{\epsilon^2+x^2}=\frac{1}{\pi}\frac{1}{x}\cdot \frac{x^2}{\epsilon^2+x^2}\\
  &=\frac{1}{\pi}\int^{|x|}_0\frac{1}{x}\frac{d}{dy}(\frac{y^2}{\epsilon^2+y^2})dy\\
  &=\frac{1}{\pi}\int^{\8}_0\frac{1}{x}\mathds{1}_{\{|x|>\epsilon y\}}\frac{2y}{(1+y^2)^2}dy.\\
\end{align*}
Therefore,
$$Q_\epsilon f(x)=\int^{\8}_{0}H_{\epsilon y}f(x)\frac{2y}{(1+y^2)^2}dy.$$
%Recalling that $$V_q(Q)f(x)=\norm{Q_tf(x)}_{V_q(X)}.$$  Then we have
Note that
\begin{align*}
 &V_q(Q_\epsilon f(x):\epsilon>0)\\
&=\sup_{0<\epsilon_0<\cdots<\epsilon_J}\bigg(\sum_{j=1}^J\norm{Q_{\epsilon_j}f(x)-Q_{\epsilon_{j-1}}f(x)}^{q}\bigg)^{\frac{1}{q}}\\
 &\leq\sup_{0<\epsilon_0<\cdots<\epsilon_J}\left(\sum_j\norm{\int^{\8}_{0}\bigg\{H_{y\epsilon_j}f(x)-H_{y\epsilon_{j-1}}f(x)\bigg\}\frac{2y}{(1+y^2)^2}dy}^q
 \right)^{\frac{1}{q}}\\
 &\leq\int^{\8}_{0}\sup_{0<\epsilon_0<\cdots<\epsilon_J}\bigg(\sum_{j}\norm{H_{y\epsilon_j}f(x)-H_{y\epsilon_{j-1}}f(x)}^{q}\bigg)^{\frac{1}{q}}\frac{2y}{(1+y^2)^2}dy\\
 &\leq\int^{\8}_{0}V_q(H_\epsilon f(x):\epsilon>0)\frac{2y}{(1+y^2)^2}dy\\
 &=C V_q(H_\epsilon f(x):\epsilon>0).
\end{align*}
By the pointwise estimate, we have
$$\norm{V_q(Q_\epsilon f:\epsilon>0)}_{L^p}\le C_p\norm{ V_q(H_\epsilon f:\epsilon>0)}_{L^p}\leq C_{p,q}\norm{f}_{L^p}.$$
Using the above inequality and equality~\eqref{the equality of Poisson function}, we obtain
$$\|V_q(P_\epsilon f:\epsilon>0)\|_{L^p}=\|V_q(Q_\epsilon(Hf):\epsilon>0)\|_{L^p}\le C_{p,q}\|Hf\|_{L^p}\le C_{p,q}\|f\|_{L^p},$$
where in the last inequality we have used that the Banach space $X$ satisfies the UMD property. Thus we obtain inequality~\eqref{variation of poisson}.
\end{proof}

Associating with Poisson kernel $P_\epsilon$, the first and third author \cite[Lemma 3.4]{GXHTM} proved the following lemma.
\begin{lem}\label{control by Poisson kernel}
Let $\mathbb{P}_\epsilon$ be the Poisson kernel of the unit disc, which was defined in section \ref{ST2}. Let $X$ be a Banach space. For $1<p<\8$, if inequality \eqref{variation of poisson} holds, then for every $f\in L^p([0,1];X)$ we have
\begin{equation}\label{the boundedness of poisson integral}
  \|V_q(\mathbb{P}_\epsilon f:\epsilon>0)\|_{L^p}\le C_{p,q}\|f\|_{L^p}.
\end{equation}
\end{lem}
With the above lemmas, we are now in a position to prove the part (ii) of Theorem~\ref{Thm:second}.
\begin{proof}[Proof of Theorem \ref{Thm:second} (ii)]
Using Lemma \ref{1 main lemma} we obtained that the Banach space $X$ satisfies the UMD property. Moreover, combining inequalities \eqref{variation of poisson} with Lemma~\ref{control by Poisson kernel}, we can establish inequalities \eqref{the boundedness of poisson integral}. In section \ref{ST2}, we proved that inequalities \eqref{the boundedness of poisson integral} implying  the Banach space $X$ is of martingale cotype $q$. This completes the proof of the part (ii) of Theorem~\ref{Thm:second}.
\end{proof}

\bigskip
%%%%%%%%%%%%%%%%%%%%%%%%%%%%%%%%%%%%%%%%%%%%%%%%%%%%%%%%%%%%%%%%%%%%%%%%%%%%%%%%%%%%%
%%%%%%%%%%%%%%%%%%%%%%%%%%%%%%%%%%%%%%%%%%%%%%%%%%%%%%%%%%%%%%%%%%%%%%%%%%%%%%%%%%%%%


\begin{thebibliography}{99}
%%%%%%%%%%%%%%%%%%%%%%%%%%%%%%%%%%%%%%%%%%%%%%%%%%%%%%%%%%%%%%%%%%%%%%%%%%%%%%%%%%%%%
%%%%%%%%%%%%%%%%%%%%%%%%%%%%%%%%%%%%%%%%%%%%%%%%%%%%%%%%%%%%%%%%%%%%%%%%%%%%%%%%%%%%%

\bibitem{Bour} J. Bourgain, \textit{Some remarks on Banach spaces in which martingale difference sequences are unconditional}, Ark. Mat. 21 (1983), no. 2, 163-168.
\bibitem{Bour1} J. Bourgain, \textit{Vector-valued singular integrals and the $H_1$-BMO duality}, Probability theory and harmonic analysis, 1-19, Textbooks Pure Appl. Math. Dekker, New York, 1986.

\bibitem{Bour2}  J. Bourgain, \textit{Pointwise ergodic theorems for arithmetic sets}, Publ. Math. IHES. 69 (1989), 5-41.

\bibitem{J.M.E.B} J. Bourgain, M. Mirek, E. M. Stein, B. Wr\'{o}bel, \textit{On dimension-free variational inequalities for averaging operators in $\real^d$}, Geom. Funct. Anal. 28 (2018), no. 1, 58-99.

\bibitem{J.M.E.B1} J. Bourgain, M. Mirek, E. M. Stein, B. Wr\'{o}bel, \textit{Dimension-free estimates for discrete Hardy-Littlewood
averaging operators over the cubes in $\mathbb{Z}^d$}, arXiv:1804.07679 (to appear).

\bibitem{Burk} D. L. Burkholder, \textit{A geometric condition that implies the existence of certain singular integrals of Banach-space-valued functions}, Conference on harmonic analysis in honor of Antoni Zygmund, Vol. I, II(Chicago, III., 1981), 270-286, Wadsworth Math. Ser., Wadsworth, Belmont, CA, 1983.

\bibitem{JRKM} J. T. Campbell, R. L. Jones, K. Reinhold, M. Wierdl, \textit{Oscillation and variation for the Hilbert transform}, {Duke Math. J.} 105 (2000), no. 1, 59-83.

\bibitem{JRKM1} J. T. Campbell, R. L. Jones, K. Reinhold, M. Wierdl, \textit{Oscillation and variation for singular integrals in higher dimensions}, Trans. Amer. Math. Soc. 355 (2003), no. 5, 2115-2137.
%\bibitem{Bour1} J. Bourgain, \textit{ Pointwise ergodic theorems for arithmetic sets}, Publ. Math. IHES. 69 (1989), 5-41.
\bibitem{CDHL} Y. Chen, Y, Ding, G. Hong, H. Liu, \textit{Weighted jump and variational inequalities for rough operators}, J. Funct. Anal. 274 (2018), no. 8, 2446-2475.

\bibitem{DHL} Y, Ding, G. Hong, H. Liu, \textit{Jump and variational inequalities for rough operators}, J. Fourier Anal. Appl. 23 (2017),
no. 3, 679-711.

\bibitem{YCC} Y. Do, C. Muscalu, C. Thiele, \textit{Variational estimates for paraproducts}, Rev. Mat. Iberoam. 28 (2012), no. 3, 8578-78.

\bibitem{AJ} A. T. Gillespie, J. L. Torrea, \textit{Dimension free estimates for the oscillation of Riesz transforms}, Isreal J. Math. 141 (2004), 125-144.

\bibitem{Graf} L. Grafakos, \textit{Classical and Modern Fourier Analysis}, Pearson/Prentice Hall, Upper-Saddle River, 2004.

\bibitem{GXHTM1} G. Hong, T. Ma, \textit{Vector valued $q$-variation for ergodic averages and analytic semigroups}, J. Math. Anal. Appl. {437} (2016), no. 2, 1084-1100.

\bibitem{GXHTM} G. Hong, T. Ma, \textit{Vector valued $q$-variation for differential operators and semigroups I}, Math. Z. 286 (2017), no. 1-2, 89-120.

\bibitem{TPH} T. P. Hyt\"{o}nen, \textit{Littlewood-Paley-Stein theory for semigroups in UMD spaces}, Rev. Mat. Iberoam. 23 (2007), no. 3, 973-1009.

\bibitem{TMC} T. P. Hyt\"{o}nen, M. T. Lacey, C. P\'erez, \textit{Sharp weighted bounds for the q-variation of singular integrals}, Bull. Lond. Math. Soc. 45 (2013), no. 3, 529-540.

\bibitem{TMI} T. P. Hyt\"{o}nen, M. T. Lacey, I. Parissis, \textit{A variation norm Carleson theorem for vector-valued Walsh-Fourier series}, Rev. Mat. Iberoam. 30 (2014), no. 3, 979-1014.

\bibitem{TJVW} T. P. Hyt\"{o}nen, V. N. Jan, M. Veraar, L. Weis, \textit{ Analysis in Banach spaces. Vol. I. Martingales and Littlewood-Paley theory}, Springer, Cham, 2016.

\bibitem{JKRW} R. L. Jones, R. Kaufman, J. M. Rosenblatt, M. Wierdl, \textit{ Oscillation in ergodic theory}. Ergodic Theory Dynam. Systems 18 (1998), no. 4, 889-935.

\bibitem{JKRW1} R. L. Jones, J. M. Rosenblatt, M. Wierdl, \textit{Oscillation inequalities for rectangles}, Proc. Am. Math. Soc. 129 1349-1358 (2000).

\bibitem{JKRW2} R. L. Jones, J. M. Rosenblatt, M. Wierdl, \textit{Oscillation in ergodic theory: higher dimensional results}, Israel J. Math. 135 (2003), 1-27.

\bibitem{JKRW3} R. L. Jones, R. Kaufman, J. M. Rosenblatt, M. Wierdl, \textit{Oscillation and variation for singular integrals in higher dimensions}, Trans. Amer. Math. Soc. 355 (2003), no. 5, 2115-2137.RG

\bibitem{RAJ} R. L. Jones, A. Seeger, J. Wright,\textit{Strong variational and jump inequalities in harmonic analysis}, Trans. Amer. Math. Soc. 360 (2008), no. 12, 6711-6742.

\bibitem{RG} R. L. Jones, G. Wang, \textit{Variation inequalities for the Fej\'{e}r and Poisson kernels}, Trans. Amer. Math. Soc. 356 (2004), no. 11, 4493-4518.

\bibitem{MX} C. Le Merdy, Q. Xu, \textit{Strong q -variation inequalities for analytic semigroups}, Ann. Inst. Fourier (Grenoble) 62 (2012), no. 6, 2069-2097 (2013).
\bibitem{K-Zk}B. Krause, P. Zorin-Kranich,\textit{Weighted and vector-valued variational estimates for ergodic averages}, Ergodic Theory Dynam. Systems 38 (2018), no. 1, 244-256.

\bibitem{DL} D. L\'{e}pingle, \textit{La variation $d^{\prime}$ordre $p$ des semi-martingales}, Z. Wahrsch. Verw. Gebiete 36 (1976), no. 4, 295-316.

\bibitem{MA} T. Ma, J. L. Torrea, Q. Xu, \textit{Weighted variation inequalities for differential operators and singular integrals}, J. Funct. Anal. 268 (2015), no. 2, 376-416.

\bibitem{MA1} T. Ma, J. L. Torrea, Q. Xu, \textit{Weighted variation inequalities for differential operators and singular integrals in higher dimensions}, Sci. China Math. 60 (2017), no. 8, 1419-1442.

%\bibitem{AM} A. Mas, \textit{Variation for singular integrals on Lipschitz graphs: $L_p$ and endpoint estimates}, Trans. Amer. Math.Soc. 365 (2013), no. 11, 5759-5781.

\bibitem{MTX}T. Mart\'{\i}nez, J. L. Torrea, Q, Xu, \textit{Vector-valued Littlewood-Paley-Stein theory for semigroups}, Adv. Math. 203 (2006), no. 2, 430-475.

\bibitem{MSB} M. Mirek, E. Stein, B. Trojan, \textit{$\ell^p(\Bbb Z^d) $-estimates for discrete operators of Radon type: variational estimates},
Invent. Math. 209 (2017), no. 3, 665-748.

\bibitem{MBP} M. Mirek, B. Trojan, P. Zorin-Kranich,  \textit{Variational estimates for averages and truncated singular integrals along the prime numbers}, Trans. Amer. Math. Soc. 369 (2017), no. 8, 5403-5423.

\bibitem{RATCJ} R. Oberlin, A. Seeger, T. Tao, C. Thiele, J. Wright, \textit{A variation norm Carleson theorem}, J. Eur. Math. Soc. (JEMS) 14 (2012), no. 2, 421-464.

\bibitem{G.P} G. Pisier, \textit{Martingale with values in uniformly convex spaces}, Israel J. Math. 20(1975), no. 3-4, 326-350.

\bibitem{G.P1} G. Pisier, Q. Xu, \textit{The strong $p$-variation of martingale and orthogonal series}, Probab. Theory Related fields  77 (1988), no. 4, 497-514.

%\bibitem{HW} H. White, \textit{The pointwise ergodic theorem and related analytic inequalities}, Master¡¯s Thesis, University of North Carolina, Chapel Hill, NC, 1989.

\bibitem{Xu} Q. Xu, \textit{Littlewood-Paley theory for functions with values in uniformly convex spaces}, J. Reine Angew. Math. 504 (1998), 195-226.




\end{thebibliography}
\end{document}